\documentclass{article}
\usepackage{spconf}
\usepackage{graphicx}

\usepackage[cmex10]{amsmath}\allowdisplaybreaks
\usepackage{amssymb}
\usepackage{epstopdf}
\usepackage{amsfonts}
\usepackage{subcaption}
\usepackage{algorithm}
\usepackage{algorithmic}
\usepackage[amsthm,thmmarks]{ntheorem}

\usepackage{color}
\usepackage{float}
\usepackage{tabularx,colortbl}
\newtheorem{thm}{Theorem}
\newenvironment{definition}[1][Definition]{\begin{trivlist}
\item[\hskip \labelsep {\bfseries #1}]}{\end{trivlist}}
\newtheorem{lemma}[thm]{Lemma}

\usepackage{defs}
\usepackage{prjohns2MacrosTA}

\title{Convergence Rates of Inertial Splitting Schemes for Nonconvex Composite Optimization}
\name{Patrick R. Johnstone and Pierre Moulin}
\address{Department of Electrical and Computer Engineering\\
University of Illinois at Urbana-Champaign\\
\small Email: prjohns2@illinois.edu, moulin@ifp.uiuc.edu}
\begin{document}
\ninept
\maketitle
\begin{abstract}
We study the convergence properties of a general inertial first-order proximal splitting algorithm for solving nonconvex nonsmooth optimization problems. Using the Kurdyka--\L ojaziewicz (KL) inequality we establish new convergence rates which apply to several inertial algorithms in the literature. Our basic assumption is that the objective function is semialgebraic, which lends our results broad applicability in the fields of signal processing and machine learning. The convergence rates depend on the exponent of the ``desingularizing function" arising in the KL inequality. Depending on this exponent, convergence may be finite, linear, or sublinear and of the form $O(k^{-p})$ for $p>1$. 
\end{abstract}
\section{Introduction}
We are interested in solving the following optimization problem
\begin{eqnarray}
\min_{x\in\re^n}\Phi(x) = f(x) + g(x)\label{prb1}
\end{eqnarray}
where $g:\re^n\to\re\cup\{+\infty\}$ is lower semicontinuous (l.s.c.) and $f:\re^n\to\re$ is differentiable with Lipschitz continuous gradient. We also assume that $\Phi$ is \emph{semialgebraic}  \cite{attouch2013convergence}, meaning there integers $p,q\geq 0$ and polynomial functions $P_{ij},Q_{ij}:\re^{n+1}\to \re$ such that
$$
\{(x,y):y\geq f(x)\} =\underset{j=1}{\overset{p}{\cup}}\underset{i=1}{\overset{q}{\cap}}\{z\in\re^{n+1}: P_{ij}(z)=0,Q_{ij}(z)<0\}.
$$
We make no assumption of convexity. Semialgebraic objective functions in the form of (\ref{prb1}) are widespread in machine learning, image processing, compressed sensing, matrix completion, and computer vision \cite{IHT,lazzaro2016nonconvex,elad2006image,chen2004recovering,zhang2006gene,candes2015phase,ji2010robust}. We will list a few examples below. 

In this paper we focus on the application of Prob. (\ref{prb1}) to \emph{sparse least-squares} and regression. This problem arises when looking for a sparse solution to a set of underdetermined linear equations. Such problems occur in compressed sensing, computer vision, machine learning and many other related fields. Suppose we observe $y=Ax+b$ where $b$ is noise and wish to recover $x$ which is known to be sparse, however the matrix $A$ is ``fat" or poorly conditioned. One approach is to solve (\ref{prb1}) with $f$ a loss function modeling the noise $b$ and $g$ a regularizer modeling prior knowledge of $x$, in this case sparsity. The correct choice for $f$ will depend on the noise model and may be nonconvex. Examples of appropriate nonconvex semialgebraic choices for $g$ are the $\ell_0$ pseudo-norm, and the smoothly clipped absolute deviation (SCAD) \cite{fan2001variable}. The prevailing convex choice is the $\ell_1$ norm which is also semialgebraic. SCAD has the advantage over the $\ell_1$-norm that it leads to nearly unbiased estimates of large coefficients \cite{fan2001variable}. Furthermore unlike the $\ell_0$ norm SCAD leads to a solution which is continuous in the data matrix $A$. Nevertheless $\ell_1$-based methods continue to be the standard throughout the literature due to convexity and computational simplicity. 

For Problem (\ref{prb1}), \emph{first-order methods} have been found to be computationally inexpensive, simple to implement, and effective solvers \cite{prox_signalProcessing}. 
In this paper we are interested in first order methods of the \emph{inertial} type, also known as \emph{momentum} methods. These methods generate the next iterate using more than one previous iterate so as to mimic the inertial dynamics of a model differential equation. In many instances both in theory and in practice, inertial methods have been shown to converge faster than noninertial ones \cite{polyak1964some}. Furthermore for nonconvex problems it has been observed that using inertia can help the algorithm escape local minima and saddle points that would capture other first-order algorithms \cite[Sec 4.1]{boct2016inertial}. A prominent example of the use of inertia in nonconvex optimization is training neural networks, which goes under the name of  \emph{back propagation with momentum} \cite{sutskever2013importance}. In convex optimization a prominent example is the heavy ball method \cite{polyak1964some}.

%Other applications include low ra 
%\item  \emph{low rank matrix completion}: In this problem one wants to complete the unknown entries of a ``low-rank" matrix. This problem is often approached by minimizing objectives involving the nonconvex and semialgebraic rank function. Other nonconvex approaches may use the fact that a low rank matrix can be factorized into the product of two smaller matrices \cite{jin2016provable} which also yields a nonconvex semialgebraic objective function. The most prominent application of matrix completion is in recommender systems although there are many others such as phase retrieval \cite{candes2015phase} and video denoising \cite{ji2010robust}. 
%\item \emph{Image processing with nonconvex objectives.} In this application it is often better to use a nonconvex loss function because it describes the true noise model more accurately than one associated with a convex loss. Other settings, such as linear diffusion based image compression, naturally lead to nonconvex objectives in the form of (\ref{prb1}) \cite{boct2016inertial,ochs2014ipiano}.
%\end{enumerate}
Over the past decade the KL inequality has come to prominence in the optimization community as a powerful tool for studying both convex and nonconvex problems. It is very general, applicable to almost all problems encountered in real applications, and powerful because it allows researchers to precisely understand the local convergence properties of first-order methods. The inequality goes back to \cite{kurdyka1998gradients,lojasiewicz1963propriete}. In \cite{attouch2010proximal,attouch2009convergence,frankel2014splitting} the KL inequality was used to derive convergence rates of descent-type first order methods. The KL inequality was used to study convex optimization problems in \cite{bolte2015error,li2016calculus}. 

Nonconvex optimization has traditionally been challenging for researchers to study since generally they cannot distinguish a local minimum from a global minimum. Nevertheless, for some applications such as empirical risk minimization in machine learning, finding a good local minimum is all that is required of the optimization solver \cite[Sec. 3]{domingos2012few}. In other problems local minima have been shown to be global minima \cite{ge2016matrix}. %Furthermore, recent work has shown that under reasonable conditions first-order methods will never converge to saddle points in practice  \cite{lee2016gradient,panageas2016gradient}. 

% The best researchers can hope for in nonconvex optimization is to develop algorithms which can find local minimizers. 

%In contrast generic worst-case convergence rates in the convex setting, such as the well-known $O(1/k)$ of gradient descent, can be pessimistic. For example consider $f(x)=x^p$ for $p>2$ and the standard gradient descent iterates $x_k=x_{k-1}-\gamma f'(x_{k-1})$ for an appropriately small $\gamma$. Then the iterates converge to $0$ at the rate $x_k = O(k^{-\frac{1}{p-2}})$, which is predicted exactly by the KL analysis \cite[Thm 4]{frankel2014splitting}. On the other hand, standard convex arguments guarantee $O(k^{-\frac{1}{p}})$ which is pessimistic. 

{\bf Contributions:} The main contribution of this paper is to determine for the first time the local convergence rate of a broad family of inertial proximal splitting methods for solving Prob. (\ref{prb1}). The family of methods we study includes several algorithms proposed in the literature for which convergence rates are unknown. The family was proposed in \cite{liang2016multi}, where it was proved that the iterates converge to a critical point. However the \emph{convergence rate}, e.g. how fast the iterates converge, was not determined. In fact in \cite{liang2016multi}, local linear convergence was shown  under a partial smoothness assumption. In contrast we do not assume partial smoothness and our results are far more general. We use the KL inequality and show finite, linear, or sublinear convergence, depending on the KL exponent (see Sec. 2). The main inspiration for our work is \cite{frankel2014splitting} which studied convergence rates of several \emph{noninertial} schemes using the KL property. However, the analysis of \cite{frankel2014splitting} cannot be applied to inertial methods. Our approach is to extend the framework of \cite{frankel2014splitting} to the inertial setting. This is done by proving convergence rates of a multistep Lyapunov potential function which upper bounds the objective function. We also prove convergence rates of the iterates and extend a result of \cite[Thm. 3.7 ]{li2016calculus} to show that our multistep Lyapunov potential has the same KL exponent as the objective function.  Finally we include experiments to illustrate the derived convergence rates.

{\bf Notation:}
Given a closed set $C$ and point $x$, define $d(x,C)\triangleq\min\{\|x-c\|:c\in C\}$. For a sequence $\{x_k\}_{k\in\mathbb{N}}$ let $\Delta_k\triangleq\|x_{k}-x_{k-1}\|$. We say that $x_k\to x^*$ linearly with convergence factor $q\in(0,1)$ if there exists $C>0$ such that $\|x_k-x^*\|\leq Cq^k$.
\section{Mathematical Background}
In this section we give an overview of the relevant mathematical concepts. 
%The \emph{Fr\'{e}chet subdifferential} of a (l.s.c.) function $\Phi:\mathbb{R}^n\to \overline{\re}$ at a point $x\in\text{dom}(\Phi)$ is defined as 
%\begin{eqnarray*}
%\partial^F \Phi(x)\triangleq \left\{v:\underset{z\to x}{\lim\inf}\left(\Phi(z)-\Phi(x)-\langle %v,z-x\rangle\geq 0\right)\right\}.
%\end{eqnarray*}
%The (limiting) subdifferential is defined as
%\begin{eqnarray*}
%\partial\Phi(x)
%\triangleq
%\{
%v:
%\exists x_k\to x,\Phi(x_k)\to\Phi(x),v_k\in\partial^F\Phi(x_k)\to v
%\}
%\end{eqnarray*}
%Note that $\partial^F \Phi(x)\subset \partial \Phi(x)$ and $\partial\Phi(x)$ is closed. 
We use the notion of the limiting subdifferential $\partial \Phi(x)$ of a l.s.c. function $\Phi$. For the definition and properties we refer to \cite[Sec 2.1]{attouch2013convergence}.
A necessary (but not sufficient) condition for $x$ to be a minimizer of $\Phi$ is $0\in\partial\Phi(x)$. The set of critical points of $\Phi$ is $\text{crit}(\Phi) \triangleq \{x:0\in\partial\Phi(x)\}$.
A useful notion is the \emph{proximal operator} w.r.t. a l.s.c. proper function $g$, defined as
\begin{eqnarray*}
\prox_{g}(x)=\underset{x'\in \mathbb{R}^n}{\arg\min}\,\,\, g(x')+\frac{1}{2}\|x-x'\|^2,
\end{eqnarray*}
which is always nonempty. Note that, unlike the convex case, this operator is not necessarily single-valued.  
\begin{definition}
A function $f:\mathbb{R}^n\to\overline{\re}$ is said to have the Kurdyka-Lojasiewicz (KL) property at $x^*\in\text{dom}\,\partial f$ if there exists $\eta\in(0,+\infty]$, a neighborhood $U$ of $x^*$, and a continuous and concave function $\varphi:[0,\eta)\to\re_+$ such that
\begin{enumerate}[(i)]
\item $\varphi(0)=0$,
\item $\varphi$ is $C^1$ on $(0,\eta)$ and for all $s\in(0,\eta)$, $\varphi'(s)>0$,
\item for all $x\in U\cap\{x: f(x^*)<f(x)<f(x^*)+\eta\}$ the KL inequality holds:
\begin{eqnarray}\label{definitiveKL}
\varphi'(f(x)-f(x^*))d(0,\partial f(x))\geq 1. 
\end{eqnarray}
\end{enumerate}
If $f$ is semialgebraic, then it has the KL property at all points in $\text{dom}\,\partial f$, and $\varphi(t)=\frac{c}{\theta}t^\theta$ for $\theta\in(0,1]$.  
\end{definition}
In the semialgebraic case we will refer to $\theta$ as the \emph{KL exponent} (note that some other papers use $1-\theta$ \cite{li2016calculus}).
For the special case where $f$ is smooth, (\ref{definitiveKL}) can be rewritten as 
$
\|\nabla(\varphi\circ (f(x)-f(x^*))\|\geq 1,
$
which shows why $\varphi$ is called a ``desingularizing function".
The slope of $\varphi$ near the origin encodes information about the ``flatness" of the function about a point, thus the KL exponent provides a way to quantify convergence rates of iterative first-order methods. 

For example the 1D function $f(x)=|x|^p$ for $p\geq2$ has desingluarizing function $\varphi(t)= t^{\frac{1}{p}}$. The larger $p$, the flatter $f$ is around the origin, and the slower gradient-based methods will converge. In general, functions with  smaller exponent $\theta$ have slower convergence near a critical point \cite{frankel2014splitting}. Thus, determining the KL exponent of an objective function holds the key to assessing convergence rates near critical points. Note that for most prominent optimization problems, determining the KL exponent is an open problem. Nevertheless many important examples have been determined recently, such as least-squares and logistic regression with an $\ell_1$, $\ell_0$, or SCAD penalty \cite{li2016calculus}. A very interesting recent work showed that for convex functions the KL property is equivalent to an error bound condition which is often easier to check in practice \cite{bolte2015error}. %We expect our understanding of the KL exponent to grow in the future.

%While determining KL exponents is a difficult problem, some progress has been made recently, which will hopefully continue in the future \cite{li2016calculus,bolte2015error}. 

%A very rich subclass of KL functions are semialgebraic functions \cite[Sec. 2.2]{attouch2013convergence}. This covers for example functions of the form $\|G(x)\|_p^q$ for $p,q\geq 0$, where $G:\re^n\to\re^m$ is any polynomial function. Notice that the $\ell_0$-pseudo norm is included. Note the rank function of a matrix is semialgebraic \cite[Example 1]{attouch2010proximal}. Semialgebraic functions possess a comprehensive calculus, which includes closure under addition and multiplication \cite[Sec. 4.3]{attouch2010proximal}. %Thus the standard optimization formulations relating to compressed sensing and matrix completion are semialgebraic \cite{}. 

We now precisely state our assumptions on Problem (\ref{prb1}), which will be in effect throughout the rest of the paper. 

{\bf Assumption 1.} The function $\Phi:\re^n\to\re\cup\{+\infty\}$ is semialgebraic, bounded from below, and has desingularizing function $\varphi(t)=\frac{c}{\theta}t^\theta$ where $c>0$ and $\theta\in(0,1]$. The function $g:\re^n\to\overline{\re}$ is l.s.c., and $f:\re^n\to\re$ has Lipschitz continuous gradient with constant $L$.

%\begin{enumerate}
%\item $\Phi$ bounded from below, nonempty domain, semialgebraic, $g$ is l.s.c., $f$ is $L$-smooth. 
%\end{enumerate} 
\section{A Family of Inertial Algorithms}
We study the family of inertial algorithms proposed in \cite{liang2016multi}. In what follows $s\geq 1$ is an integer, and $I=\{0,1,\ldots,s-1\}$.

 \begin{algorithm}
  \caption{Multi-step Inertial Forward-Backward splitting (MiFB)}
  \begin{algorithmic}
  \REQUIRE $x_0\in\re^n$, $0<\underline{\gamma}\leq\overline{\gamma}<1/L$. 
  \STATE Set $x_{-s}=\ldots=x_{-1}=x_0$, $k=1$
  \REPEAT
   \STATE Choose $0<\underline{\gamma}\leq\gamma_k\leq\overline{\gamma}<1/L$, $\{a_{k,0}, a_{k,1},\ldots\} \in(-1,1]^s$, $\{b_{k,0}, b_{k,1},\ldots\} \in(-1,1]^s$.
   \STATE $y_{a,k} = x_k +\sum_{i\in I}a_{k,i}(x_{k-i}-x_{k-i-1})$
   \STATE $y_{b,k} = x_k +\sum_{i\in I}b_{k,i}(x_{k-i}-x_{k-i-1})$
   \STATE $x_{k+1}\in\prox_{\gamma_k g}\left(y_{a,k}-\gamma_k\nabla f(y_{b,k})\right)$
   \STATE  $k=k+1$
    \UNTIL{convergence}
  \end{algorithmic}
  \end{algorithm}
  Note the algorithm as stated leaves open the choice of the parameters $\{a_{k,i},b_{k,i},\gamma_k\}$. For convergence conditions on the parameters we refer to Section \ref{secConv}  and \cite[Thm. 1]{Johnstone2016Local}.
  
   The algorithm is very general and covers several inertial algorithms proposed in the literature as special cases. For instance the inertial forward-backward method proposed in \cite{boct2016inertial} corresponds to MiFB with $s=1$, and $b_{k,0}=0$. The well-known iPiano algorithm also corresponds to this same parameter choice, however the original analysis of this algorithm assumed $g$ was convex \cite{ochs2014ipiano}. The \emph{heavy-ball method} is an early and prominent inertial first-order method which also corresponds to this parameter choice when $g=0$. The heavy-ball method was originally proposed for strongly convex quadratic problems but was considered in the context of nonconvex problems in \cite{zavriev1993heavy}. The analysis of \cite{xu2013block} applies to MiFB for the special case when $s=1$ and $a_{k,0}=b_{k,0}$. However \cite{xu2013block} only derived convergence rates of the iterates and not the function values, which are our main interest\footnote{Note that the objective function is not Lipschitz continuous so rates derived for the iterates do not immediately imply rates for the objective.}. Furthermore \cite{xu2013block} used a different proof technique to the one used here. This same parameter choice has been considered for convex optimization in \cite{Johnstone2016Local,lorenz2014inertial}, albeit without the sharp convergence rates derived here. In both the convex and nonconvex settings, employing inertia has been found to improve either the convergence rate or the quality of the obtained local minimum in several studies \cite{boct2016inertial,ochs2014ipiano,liang2016multi,Johnstone2016Local}. 
   %We mention that it was noticed in \cite{liang2016multi} that using multiple inertial terms (i.e. $s>1$) can lead to improved performance. 
   
   General convergence rates have not been derived for MiFB under nonconvexity and semialgebraicity assumptions. The convergence rate of iPiano has been examined in a limited situation where the KL exponent $\theta=1/2$ in \cite[Thm 5.2]{li2016calculus}. Note that the primary motivation for studying this framework is its generality - allowing our analysis to cover many special cases from the literature. However the case $s=1$ is the most interesting in practice and corresponds to the most prominent inertial algorithms. 
   
\section{Convergence Rate Analysis}\label{secConv}
Throughout the analysis, Assumption 1 is in effect.  
Before providing our convergence rate analysis, we need a few results from \cite{liang2016multi}.  
\begin{thm}\label{liangsThm}
Fix $s\geq 1$ and recall $I=\{0,1,\ldots, s-1\}$. Fix $\{\gamma_k\}$, $\{a_{k,i}\}$ and $\{b_{k,i}\}$ for $k\in\mathbb{N}$ and $i\in I$. Fix $\mu,\nu>0$ and define
\begin{eqnarray*}
\beta_k &\triangleq& \frac{1-\gamma_k L - \mu - \nu\gamma_k}{2\gamma_k},\quad \underline{\beta}\triangleq\underset{k\in\mathbb{N}}{\lim\inf}\,\beta_k,\label{para1}
\\\label{para2}
\alpha_{k,i}&\triangleq&
\frac{s a_{k,i}^2}{2\gamma_k\mu}+\frac{s b_{k,i}^2 L^2}{2\nu},\quad \overline{\alpha}_i\triangleq\underset{k\in\mathbb{N}}{\lim\sup}\,\alpha_{k,i},
\end{eqnarray*}
and $z_k \triangleq (x_k^\top,x_{k-1}^\top,\ldots,x_{k-s}^\top)^\top$. Define the multi-step Lyapunov function as 
\begin{eqnarray}
\Psi(z_k)\triangleq \Phi(x_k)+\sum_{i\in I}\left(\sum_{j=i}^{s-1}\overline{\alpha}_j\right)\Delta_{k-i}^2.
\label{defLyap}
\end{eqnarray}
and
\begin{eqnarray}
\delta\triangleq \underline{\beta}-\sum_{i\in I}\overline{\alpha}_i>0.\label{deltaDef}
\end{eqnarray}
If the parameters are chosen so that $\delta>0$ 
%to satisfy \cite[eq. (2.2)--(2.3)]{liang2016multi}. Then 
then
\begin{enumerate}[(i)]
\item for all $k$, $\Psi(z_{k+1})\leq\Psi(z_k) - \delta\Delta_{k+1}^2$,
%\begin{eqnarray}\label{descent}
%
%\end{eqnarray}
\item for all $k$, there is a $\sigma>0$ such that $d(0,\partial\Psi(z_k))
\leq
\sigma\sum_{j=k+1-s}^{k}\Delta_j$,
%\begin{eqnarray}\label{subgradientIneq}
%d(0,\partial\Psi(z_k))
%\leq
%\sigma\sum_{j=k+1-s}^{k}\Delta_j,
%\end{eqnarray}
\item If $\{x_k\}$ is bounded there exists $x^*\in\text{crit}(\Phi)$ such that $x_k\to x^*$ and $\Phi(x_k)\to\Phi(x^*)$.
\end{enumerate}
\end{thm} 
\begin{proof}
Statements (i) and (ii) are shown in \cite[Lemma A.5]{liang2016multi} and \cite[Fact (R.2)]{liang2016multi} respectively. The fact that $\Phi(x_k)\to\Phi(x^*)$ is shown in \cite[Lemma A.6]{liang2016multi}. The fact that $x_k\to x^*$ is the main result of \cite[Thm 2.2]{liang2016multi}. 
\end{proof}
The assumption that $\{x_k\}$ is bounded is standard in the analysis of algorithms for nonconvex optimization and is guaranteed under ordinary conditions such as coercivity. Since the set of semialgebraic functions is closed under addition, $\Psi$ is semialgebraic \cite{bolte2014proximal}. 
We now give our convergence result. 
\begin{thm}
Assume the parameters of MiFB are chosen such that $\delta>0$ where $\delta$ is defined in (\ref{deltaDef}), thus there exists a critical point $x^*$ such that $x_k\to x^*$. Let $\theta$ be the KL exponent of $\Psi$ defined in (\ref{defLyap}). 
\begin{enumerate}[(a)]
\item If $\theta=1$, then $x_k$ converges to $x^*$ in a finite number of iterations. 
\item If $\frac{1}{2}\leq\theta<1$, then $\Phi(x_k)\to \Phi(x^*)$ linearly. 
\item If $0<\theta<1/2$, then $\Phi(x_k)-\Phi(x^*) = O\left(k^{\frac{1}{2\theta - 1}}\right)$.
\end{enumerate}\label{ThmOurRate}
\end{thm}
\vspace{-0.5cm}
\begin{proof} 
The starting point is the KL inequality applied to the multi-step Lyapunov function defined in (\ref{defLyap}). Let $z^*\triangleq((x^*)^\top,\ldots,(x^*)^\top)^\top$. Suppose $\Psi(z_K)=\Psi(z^*)$ for some $K>0$. Then the descent property of Thm.~1(i), along with the fact that $\Psi(z_k)\to\Psi(z^*)$, implies that $\Delta_{K+1}=0$ and therefore $\Psi(z_k)=\Psi(z^*)$ holds for all $k>K$. Therefore assume $\Psi(z_k)>\Psi(z^*)$. Now since $z_k\to z^*$ and $\Psi(z_k)\to\Psi(z^*)$, there exists $k_0>0$ such that for $k>k_0$ (\ref{definitiveKL}) holds with $f=\Psi$. Assume $k>k_0$. Squaring both sides of (\ref{definitiveKL}) yields
\begin{eqnarray}\label{KLeq}
\varphi'^2(\Psi(z_k)-\Psi(z^*))d(0,\partial\Psi(z_k))^2\geq 1,
\end{eqnarray}
Now substituting Thm.1~(ii) into (\ref{KLeq}) yields
\begin{eqnarray}
\sigma^2\varphi'^2(\Psi(z_k)-\Psi(z^*))
\left(\sum_{j=k+1-s}^{k}\Delta_j\right)^2
\geq 1.\label{KLeq2}
\end{eqnarray}
Now  
\begin{eqnarray}\nonumber
\left(\sum_{j=k+1-s}^{k}\Delta_j\right)^2
&\leq&
s\sum_{j=k+1-s}^{k}\Delta_j^2
\\
&\leq &
\frac{s}{\delta}\sum_{j=k+1-s}^{k}
\left(
\Psi(z_{j-1})-\Psi(z_j)
\right)\nonumber
\\
&=&
\frac{s}{\delta}
\left(
\Psi(z_{k-s}) - \Psi(z_k)
\right),\nonumber
\end{eqnarray}
where in the first inequality we have used the fact that $(\sum_{i=1}^s a_i)^2\leq s\sum_{i=1}^n a_i^2$, and in the second inequality we have used Thm.~1(i). Substituting this into (\ref{KLeq2}) yields
\begin{eqnarray*}
\frac{\sigma^2 s}{\delta}
\varphi'^2(\Psi(z_k)-\Psi(z^*))
\left(
\Psi(z_{k-s}) - \Psi(z_k)
\right)
\geq 1,\label{so}
\end{eqnarray*}
from which convergence rates can be derived by extending the arguments in \cite[Thm 4]{frankel2014splitting}. 

Proceeding, let $r_k \triangleq \Psi(z_k)-\Psi(z^*)$, and  $C_1\triangleq \frac{\delta}{\sigma^2 c^2 s}$, then using $\varphi'(t) = c t^{\theta - 1}$, we get
\begin{eqnarray}\label{mainRecursion}
r_{k-s}-r_k\geq C_1r_k^{2(1-\theta)}.
\end{eqnarray}
If $\theta=1$, then the recursion becomes
$
r_{k-s}-r_k\geq C_1,\quad \forall k>k_0.
$
Since by Theorem \ref{liangsThm} (iii), $r_k$ converges, this would require $C_1=0$, which is a contradiction. Therefore there exists $k_1$ such that $r_k=0$ for all $k>k_1$. 

Suppose $\theta\geq  1/2$, then since $r_k\to 0$, there exists $k_2$ such that for all $k>k_2$, $r_k\leq 1$, and $r_k^{2(1-\theta)}\geq r_k$. Therefore for all $k>k_2$,
\begin{eqnarray}
r_{k-s} - r_k \geq C_1 r_k
\implies
r_k
&\leq& (1+C_1)^{-1}r_{k-s}
\nonumber\\\label{eq_linear}
&\leq&
(1+C_1)^{-p_1}r_{k_2},
\end{eqnarray}
where $p_1\triangleq\lfloor\frac{k-k_2}{s}\rfloor$ Note that $p_1> \frac{k-k_2-s}{s}$.
Therefore $r_k\to 0$ linearly. Note that if $\theta=\frac{1}{2}$, $2(1-\theta)=1$ and (\ref{eq_linear}) holds for all $k\geq k_0$. 

Finally suppose $\theta< 1/2$. 
Define $\phi(t) \triangleq \frac{D}{1-2\theta}t^{2\theta - 1}$ where $D>0$, so $\phi'(t) = -D t^{2\theta - 2}$. Now
\begin{eqnarray*}
\phi(r_{k})-\phi(r_{k-s}) = \int_{r_{k-s}}^{r_k}\phi'(t)dt
&=&
D\int^{r_{k-s}}_{r_k}
t^{2\theta - 2}dt.
\end{eqnarray*}
Therefore since $r_{k-s}\geq r_k$ and $t^{2\theta-2}$ is nonincreasing, 
\begin{eqnarray*}
\phi(r_{k})-\phi(r_{k-s})\geq D(r_{k-s}-r_k)r_{k-s}^{2\theta - 2}.
\end{eqnarray*}
Now we consider two cases. 

{\bf Case 1:} suppose $2 r_{k-s}^{2\theta - 2}\geq r_k^{2\theta - 2}$, then 
\begin{eqnarray}
\phi(r_k) -\phi(r_{k-s})
\geq 
\frac{D}{2}
(r_{k-s}-r_k)
r_k^{2\theta - 2}
\geq 
\frac{C_1D}{2}.
\label{lowerForPhi1}
\end{eqnarray} 
where in the second inequality we have used (\ref{mainRecursion}). 

{\bf Case 2:} suppose that $2 r_{k-s}^{2\theta - 2}< r_k^{2\theta - 2}$. Now $2\theta - 2<2\theta - 1 < 0$, therefore $(2\theta - 1)/(2\theta - 2)>0$, thus
$
r_k^{2\theta - 1}> q r_{k-s}^{2\theta-1}
$
where $q = 2^{\frac{2\theta - 1}{2\theta - 2}}>1$. Thus
\begin{eqnarray} 
\phi(r_k) - \phi(r_{k-s})
&=&
\frac{D}{1-2\theta}\left(
r_k^{2\theta - 1} - r_{k-s}^{2\theta - 1}
\right)
\nonumber\\
&>&
\frac{D}{1-2\theta}
(q - 1)r_{k-s}^{2\theta - 1}
\nonumber\\
&\geq&
\frac{D}{1-2\theta}
(q - 1)r_{k_0}^{2\theta - 1} \triangleq C_2\label{lowerForPhi2}.
\end{eqnarray}
Thus putting together (\ref{lowerForPhi1}) and (\ref{lowerForPhi2}) yields
$
\phi(r_k)\geq \phi(r_{k-s})+C_3
$
where $C_3 = \max(C_2,\frac{C_1D}{2})$. 
Therefore
\begin{eqnarray*}
\phi(r_k)&\geq& \phi(r_k)-\phi(r_{k-p_2 s})\geq
p_2C_3
\end{eqnarray*}
where $p_2\triangleq\lfloor\frac{k-k_0}{s}\rfloor$.
Therefore 
\begin{eqnarray*}
r_k \leq \left(\frac{1-2\theta}{D}\right)^{\frac{1}{2\theta - 1}}{(p_2C_3)}^{\frac{1}{2\theta - 1}}
\leq 
C_4 \left(\frac{k-s-k_0}{s}\right)^{\frac{1}{2\theta-1}}.
\end{eqnarray*} 
where $C_4 = \left(\frac{C_3(1-2\theta)}{D}\right)^{\frac{1}{2\theta - 1}}$.
To end the proof, note that $\Phi(x_k)\leq \Psi(z_k)$. 
\end{proof}
%REMOVE? We mention that even though Theorem \ref{ThmOurRate} shows that the convergence rate of inertial splitting is the same (up to constants) as the rates derived in \cite{frankel2014splitting} for noninertial methods, this does not mean they have the same overall performance. In fact, one of the main advantages of inertial methods in nonconvex optimization is the ability to escape ``bad" local minima and saddle points in a principled way \cite[Sec 4.1]{boct2016inertial},\cite{sutskever2013importance}.
In the case where $f$ and $g$ are also convex, we can use parameter choices specified in \cite[Thm. 1]{Johnstone2016Local}.

\section{Convergence Rates of the Iterates}
The convergence rates of $\|x_k - x^*\|$ can also be quantified. To do so we need another result from \cite{liang2016multi}.
\begin{lemma}
 Recall the notation $r_k\triangleq \Psi(z_k)-\Psi(z^*)$. Let $\epsilon_k\triangleq\frac{\sigma}{\delta}\left(\varphi(r_k)-\varphi(r_{k+1})\right)$ where $\sigma$ is defined in Theorem \ref{liangsThm} (ii) and $\delta$ in (\ref{deltaDef}). Fix $\kappa>0$ so that $\kappa<2/s$.  
Assume the parameters of MiFB are chosen to so that $\delta>0$ and $\{x_k\}$ is bounded. Then there exists a $k_0>0$ such that for all $k>k_0$
\begin{eqnarray}
r_k>0\implies \Delta_k\leq\frac{\kappa}{2}\sum_{j=k-s}^{k-1}\Delta_j + \frac{1}{2\kappa}\epsilon_{k-1}.\label{iterate}
\end{eqnarray}
\end{lemma}
\begin{proof}
This inequality is proved on page 14 of \cite{liang2016multi} as part of the proof of \cite[Thm 2.2]{liang2016multi}.
\end{proof}
We now state our result. 
\begin{thm}
Assume $\{x_k\}$ is bounded and the parameters of MiFB are chosen so that $\delta>0$ where $\delta$ is defined in (\ref{deltaDef}). Let $\theta$ be the KL exponent of $\Psi$ defined in (\ref{defLyap}). Then 
\begin{enumerate}[(a)]
\item If $\theta=1$, then $x_k=x^*$ after finitely many iterations.
\item If $\frac{1}{2}\leq \theta<1$, $x_k\to x^*$ linearly. 
\item If $0<\theta<\frac{1}{2}$, $\|x_k-x^*\| = O\left(k^{\frac{\theta}{2\theta-1}}\right)$. 
\end{enumerate}\label{ThmIterates}
\end{thm}

% Then 
%\begin{eqnarray*}
%\|x_k - x^*\| = O\left(\tilde{\varphi}(r_{k-s-1})\right)
%\end{eqnarray*}
%where $\tilde{\varphi}(t)=\max\{\varphi(t),\sqrt{t}\}$. \label{thmIterates}
\begin{proof}
Statement (a) follows trivially from the fact that $r_k=0$ after finitely many iterations, and therefore $\Delta_k=0$. We proceed to prove statements (b) and (c). 
As with Theorem \ref{ThmOurRate} the basic idea is to extend the techniques of \cite{frankel2014splitting} to allow for the inertial nature of the algorithm. 
The starting point is (\ref{iterate}). Fix $K>k_0$. Then
\begin{eqnarray*}
\sum_{k\geq K}\Delta_k
&\leq&
\frac{\kappa}{2}\sum_{k\geq K}\sum_{j=k-s}^{k-1}\Delta_j + \frac{1}{2\kappa}\sum_{k\geq K}\epsilon_{k-1}
\\
&\leq&
\frac{\kappa s}{2}\sum_{k\geq K-s} \Delta_k + \frac{1}{2\kappa}\sum_{k\geq K}\epsilon_{k-1}
\end{eqnarray*}
Let $C=\frac{\kappa s}{2}$ and note that $0<C<1$. Therefore subtracting $C\sum_{k\geq K}\Delta_k$ from both sides yields
\begin{eqnarray*}
\sum_{k\geq K}\Delta_k
\leq 
\frac{1}{1-C}\left(
C\sum_{k=K-s}^{K-1}\Delta_k + \frac{1}{2\kappa}\sum_{k\geq K}\epsilon_{k-1}
\right).
\end{eqnarray*}
Next note that
\begin{eqnarray*}
\sum_{k=K-s}^{K-1}
\Delta_k
&\leq&
\sqrt{\frac{s}{\delta}}
\left(
\Psi(z_{K-s-1})-\Psi(z_{K-1})
\right)^{1/2}
\\
&\leq&
\sqrt{\frac{s}{\delta}}
\sqrt{r_{K-s-1}}
\end{eqnarray*}
Let $C'\triangleq C\sqrt{\frac{s}{\delta}}$ then using $\sum_{k\geq K}\epsilon_{k-1}= \frac{\sigma}{\delta}\varphi(r_{K-1})$,
\begin{eqnarray*}
\sum_{k\geq K}
\Delta_k
&\leq&
\frac{1}{1-C}\left(
C'\sqrt{r_{K-s-1}}
+\frac{\sigma}{\delta}\varphi(r_{K-1})
\right)
\\
&\leq&
\frac{1}{1-C}\left(
C'\sqrt{r_{K-s-1}}
+\frac{\sigma}{\delta}\varphi(r_{K-s-1})
\right),
\end{eqnarray*}
where in the second inequality we used the fact that $r_k$ is nonincreasing and $\varphi$ is a monotonic increasing function.
Thus using the triangle inequality and the fact that $\lim_{k}\|x_k-x^*\|=0$, 
\begin{eqnarray*}
\|x_K-x^*\|
\leq 
\sum_{k\geq K}\Delta_k
\leq
\frac{1}{1-C}\left(
C'\sqrt{r_{K-s-1}}
+\frac{\sigma}{\delta}\varphi(r_{K-s-1})
\right).
\end{eqnarray*}
Hence if $r_k\to 0$ linearly, then so does $\|x_k-x^*\|$, which proves (b). On the other hand if $0<\theta<1/2$, for $k$ sufficiently large we see that $\|x_k-x^*\|=O(\varphi(r_{k-s-1}))$, which proves statement (c). 
\end{proof} 
%Thus we can explicate the convergence rates of the iterates. 
%\begin{corollary}
%Assume the parameters of MiFB are chosen to satisfy (\ref{para1})--(\ref{para3}) and $\{x_k\}$ is bounded.
%\end{corollary}

\section{KL Exponent of the Lyapunov Function}\label{secKLexp}
We now extend the result of \cite[Thm 3.7]{li2016calculus} so that it covers the Lyapunov function defined in (\ref{defLyap}).

\begin{thm}
Let $s\geq 1$ and consider
\begin{eqnarray}\label{ALyap}
\Psi^{(s)}(x_1,x_2,\ldots,x_s)\triangleq\Phi(x_1)+\sum_{i=1}^{s-1}c_i\|x_{i+1}-x_i\|^2.
\end{eqnarray}

If $\Phi$ has KL exponent $\theta\in(0,1/2]$ at $\barx$ then $\Psi^{(s)}$ has KL exponent $\theta$ at $[\barx,\barx,\ldots,\barx]^\top$.
\end{thm}
\begin{proof} 
Before commencing, note that if $\Phi$ has desingularizing function $\varphi(t)=\frac{c}{\theta}t^\theta$, the KL inequality (\ref{definitiveKL}) can be written in the equivalent ``error bound" form:
\begin{eqnarray*}
d(0,\partial\Phi(x))^{1/\alpha}\geq c'(\Phi(x)-\Phi(x^*))
\end{eqnarray*}
where $c'\triangleq c^{-1}>0$, and $\alpha\triangleq 1-\theta$. We now show that this error bound holds for the Lyapunov function in (\ref{ALyap}).

The key is to notice the recursive nature of the Lyapunov function. In particular for all $s\geq 2$
\begin{eqnarray*}
\Psi^{(s)}(x^1_s)
&=&
\Psi^{(s-1)}(x^{s-1}_1)
\\
&+&c_{s-1}\|x_{s-1}-x_{s}\|^2,
\end{eqnarray*}
with $\Psi^{(1)}(x_1^1)\triangleq\Phi(x_1)$, and $x^{s}_1\triangleq[x_1^\top,\ldots,x_{s}^\top]^\top$. 
Since $\Phi$ has KL exponent $\theta$ at $\barx$, $\Psi^{(1)}$ has KL exponent $\theta$ at $\barx$. We will prove the following inductive step for $s\geq 2$: If $\Psi^{(s-1)}$ has KL exponent $\alpha$ (with constant $c'$) at $\barx^{s-1}_1$, then $\Psi^{(s)}$ has KL exponent $\alpha$ at $\barx^{s}_1$ where $\bar{x}^{s}_1\triangleq[\barx,\barx,\ldots,\barx]^\top$ where $\barx$ is repeated $s$ times.

Proceeding, for $s\geq 2$ assume $x_1,x_2,\ldots,x_s$ are such that $\|x_{s}-x_{s-1}\|\leq 1$ and the KL inequality (\ref{definitiveKL}) applies to $\Psi^{(s-1)}$ at $\barx^{s}_1$.  Then
\begin{eqnarray*}
\partial\Psi^{(s)}(\barx_s^1)
\ni
\left(
\begin{array}{c}
\xi^{s-2}_1\\
\xi_{s-1}\\
0
\end{array}
\right)
+
\left(
\begin{array}{c}
0\\
c_{s-1}(x_{s-1}-x_{s})\\
c_{s-1}(x_{s}-x_{s-1})
\end{array}
\right)
\end{eqnarray*}
where $(\xi^{s-2}_1,\xi_{s-1})\in\partial\Psi^{(s-1)}(x^{s-2}_1,x_{s-1})$. Therefore
\begin{eqnarray*}
&& d(0,\partial\Psi^{(s)}(x^{s}_1))^{1/\alpha}
\\
&\overset{(a)}{\geq}&
C_1\left(
\inf_{(\xi^{s-2}_1,\xi_{s-1})\in\partial\Psi^{(s-1)}(x^{s-2}_1,x_{s-1})}
\|\xi_{s-1}^c\|^{1/\alpha}
\right. 
\\
&&
\left.
+
\|\xi_{s-1}
+c_{s-1}(x_{s-1}-x_{s})
\|^{1/\alpha}
+
\|c_{s-1}(x_{s}-x_{s-1})\|^{1/\alpha}
\right)
\\
&\overset{(b)}{\geq}&
C_1\left(
\inf_{(\xi_{s-1}^c,\xi_{s-1})\in\partial\Psi^{(s-1)}(x^{s-2}_1,x_{s-1})}
\|\xi_{s-1}^c\|^{1/\alpha}
+
\eta_1\|\xi_{s-1}\|^{1/\alpha}
\right. 
\\
&&
\left.
-
\eta_2\|c_{s-1}(x_{s-1}-x_{s})\|^{1/\alpha}
+
\|c_{s-1}(x_{s}-x_{s-1})\|^{1/\alpha}
\right)
\\
&\overset{(c)}{\geq}&
C_2\left(
\inf_{(\xi_{s-1}^c,\xi_{s-1})\in\partial\Psi^{(s-1)}(x^{s-2}_1,x_{s-1})}
\|\xi_{s-1}^c\|^{1/\alpha}
+
\|\xi_{s-1}\|^{1/\alpha}
\right. 
\\
&&
\left. 
+
\frac{c_{s-1} c'}{2}
\|x_{s}-x_{s-1}\|^{1/\alpha}
\right)
\\
&\overset{(d)}{\geq}&
C_3
\left(
\inf_{(\xi_{s-1}^c,\xi_{s-1})\in\partial\Psi^{(s-1)}(x^{s-2}_1,x_{s-1})}
\left\|
\begin{array}{c}
\xi_{s-1}^c\\
\xi_{s-1}
\end{array}
\right\|^{1/\alpha}
\right. 
\\
&&
\left. 
+
\frac{c_{s-1} c'}{2}
\|x_{s}-x_{s-1}\|^{1/\alpha}
\right)
\\
&\overset{(e)}{\geq}&
C_3 c'
\left(
\Psi^{(s-1)}(x^{s-1}_1)-\Psi^{(s-1)}(\barx^{s-1}_1)
\right. 
\\
&&
\left.
+
\frac{c_{s-1}}{2}
\|
x_{s}-x_{s-1}
\|^{1/\alpha}
\right)
\\
&\overset{(f)}{\geq}&
C_3 c'
\left(
\Psi^{(s-1)}(x^{s-1}_1)-\Psi^{(s-1)}(\barx^{s-1}_1)
\right. 
\\
&&
\left.
+
\frac{c_{s-1}}{2}
\|
x_{s}-x_{s-1}
\|^{2}
\right)
\\
&=&
C_3 c'\left(\Psi^{(s)}(x^{s}_1)-\Psi^{(s)}(\barx^{s}_1)\right).
\end{eqnarray*}
Now (a) and (d) follow from \cite[Lemma 2.2]{li2016calculus}, and  (b) follows from \cite[Lemma 3.1]{li2016calculus}. Next (c) follows because $\eta_1>0$, $0<\eta_2<1$, and we have decreased $C_2$ to compensate for factoring out these coefficients. Further (e) follows by the KL inequality. Finally (f) follows because $\|x_s-x_{s-1}\|\leq 1$ and $\alpha^{-1}\in(1,2]$.  Since $\Psi^{(1)}$ has KL exponent $\alpha$ at $\barx$, then so does $\Psi^{(s)}$ at $[\barx,\barx,\ldots,\barx]^\top$ (of length $s$) for all $s\geq 2$, which concludes the proof.
\end{proof}

The power of this theorem is that when the KL exponent of the objective function $\Phi$ is known, it also applies to the Lyapunov function in (\ref{defLyap}). This allows us to exactly determine the convergence rate of MiFB via Theorems \ref{ThmOurRate} and \ref{ThmIterates}. 

\section{Numerical Results}

\subsection{One Dimensional Polynomial}
This simple experiment verifies the convergence rates derived in Theorem \ref{ThmOurRate} for MiFB. Consider the one dimensional function $f(x)=|x|^p$ for $p>2$. Use $g(x)=+\infty$ if $|x|>1$ and $0$ otherwise. The proximal operator is simple projection and $f$ is $p(p-1)$-smooth on this set. The function $\Phi=f+g$ is semialgebraic with $\varphi(t) =p t^{1/p}$, i.e. $\theta=1/p$. Therefore Theorem 2 predicts $O\left(k^{-\frac{p}{p-2}}\right)$ rates for MiFB, which is verified in Fig. \ref{figPoly} for three parameter choices in the cases $p=4,18$. For simplicity we ignore constants and focus on the sublinear order. For $p\leq 4$ this convergence rate is better than that of Nesterov's accelerated method \cite{nesterov2004introductory}, for which only $O(1/k^2)$ worst-case rate is known. Faster rates are achievable due to the additional knowledge of the KL exponent.
\begin{figure}[!h]
\centering
\includegraphics[width=1.6in]{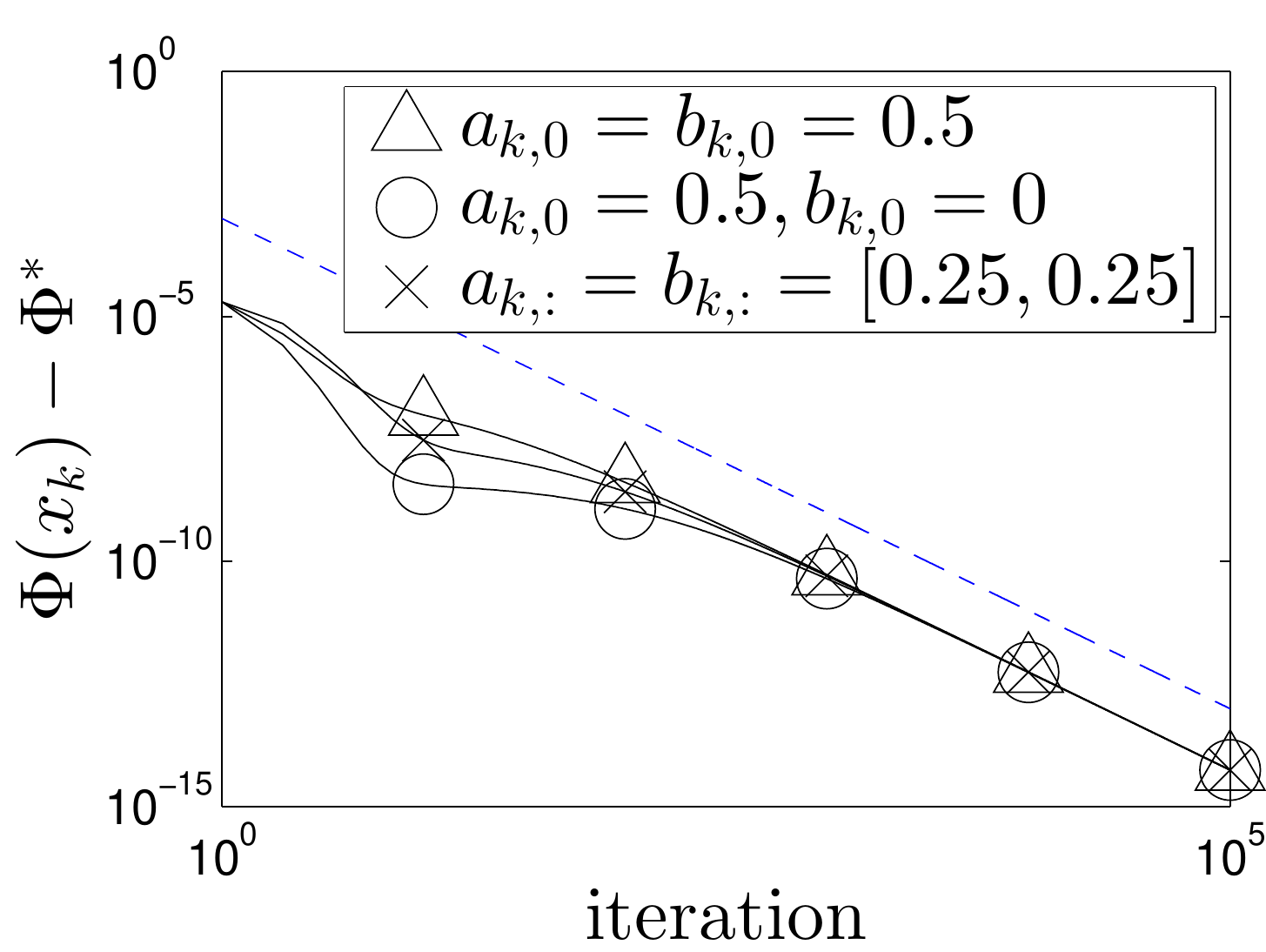}
\includegraphics[width=1.6in]{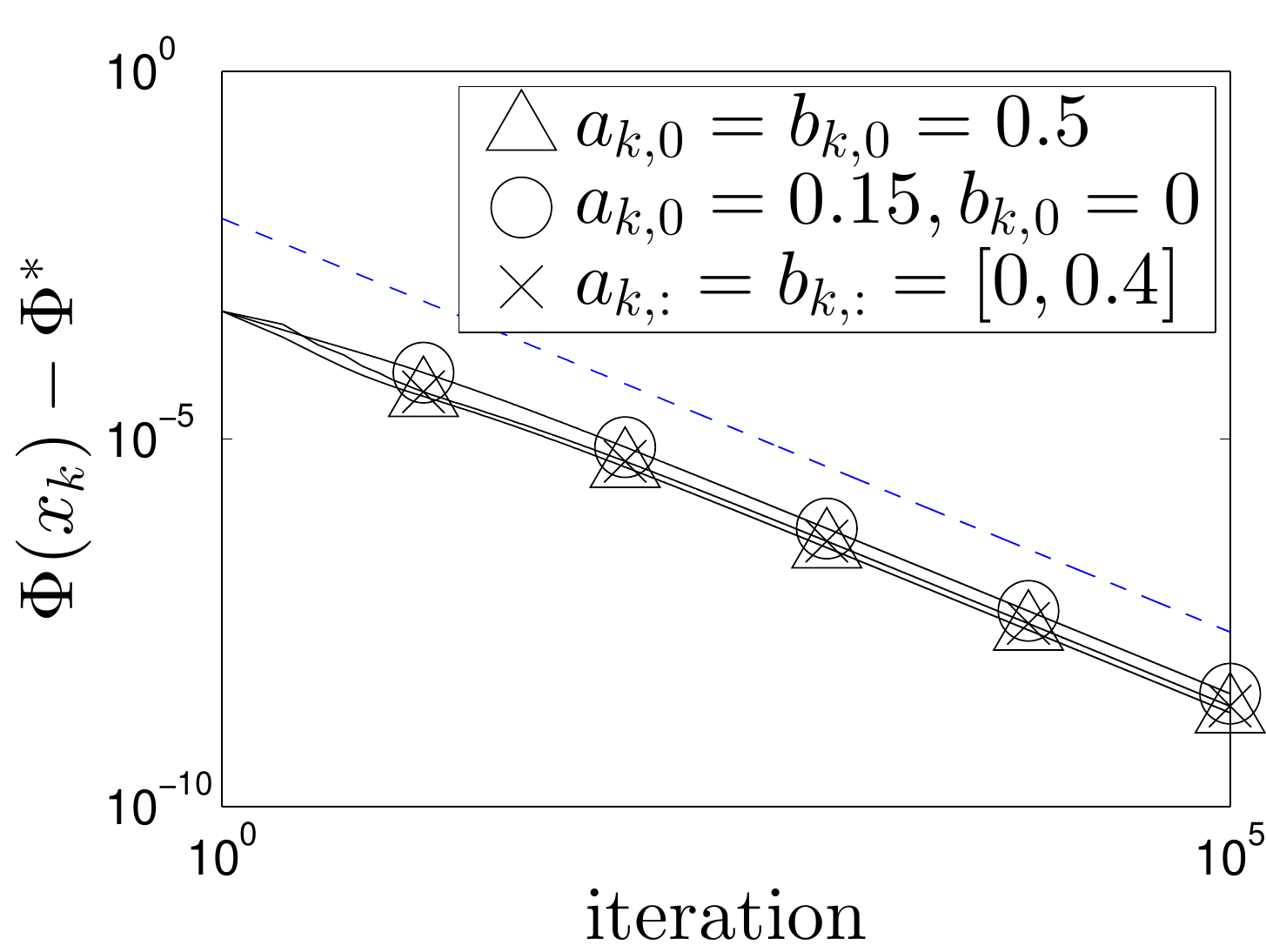}
% where an .eps filename suffix will be assumed under latex, 
% and a .pdf suffix will be assumed for pdflatex; or what has been declared
% via \DeclareGraphicsExtensions.
\caption{(Left) $p=4$, (Right) $p=18$, $\Phi^*=0$. The dotted line is the slope of the predicted $O\left(k^{-\frac{p}{p-2}}\right)$ rate (i.e. ignoring constants). Note $a_{k,:}\triangleq [a_{k,0},a_{k,1}]$ and these are log-log plots.}
\label{figPoly}
\end{figure}

\subsection{SCAD and $\ell_1$ regularized Least-Squares}
 We solve Prob. (\ref{prb1}) with $f(x)=\frac{1}{2}\|Ax - b\|_2^2$ and $g(x)=\sum_{i=1}^n r(x_i)$ where $r$ is: 1) the SCAD regularizer defined as
\begin{eqnarray*}
r(x_i)=
\left\{
\begin{array}{ll}
\lambda|x_i| & \text{if }|x_i|\leq\lambda\\
-\frac{|x_i|^2 - 2a\lambda|x_i|+\lambda^2}{2(a-1)} & \text{if }\lambda<|x_i|\leq a\lambda \\
\frac{(a+1)\lambda^2}{2} & \text{if }|x_i|>a\lambda,
\end{array}
\right.
\end{eqnarray*}
and 2) the absolute value $r(x_i)=\lambda|x_i|$ leading to the $\ell_1$-norm. 
%is the SCAD regularizer often used in regression \cite{fan2001variable}. 
In both cases the proximal operator w.r.t. $r$ is easily computed. 
It was shown in \cite[Sec. 5.2]{li2016calculus} and \cite[Lemma 10]{bolte2015error} that both of these objective functions are KL functions with exponent $\theta=1/2$. 
%Note in the case of the $\ell_1$ norm, since the objective function is convex the analysis of \cite[Thm. 1]{Johnstone2016Local} implies weaker restrictions on the parameters than \cite[eq. (2.2)--(2.3)]{liang2016multi}. 

We choose $A\in\re^{500\times 1000}$ having i.i.d. $\mathcal{N}(0,10^{-4})$ entries, and $b = Ax_0$, where $x_0\in\re^{1000}$ has $50$ nonzero $\mathcal{N}(0,1)$-distributed entries. For SCAD we use $a=5$ and $\lambda=1$ and for the $\ell_1$ norm we use $\lambda=0.01$.

\begin{figure}[!h]
\centering
\includegraphics[width=1.6in]{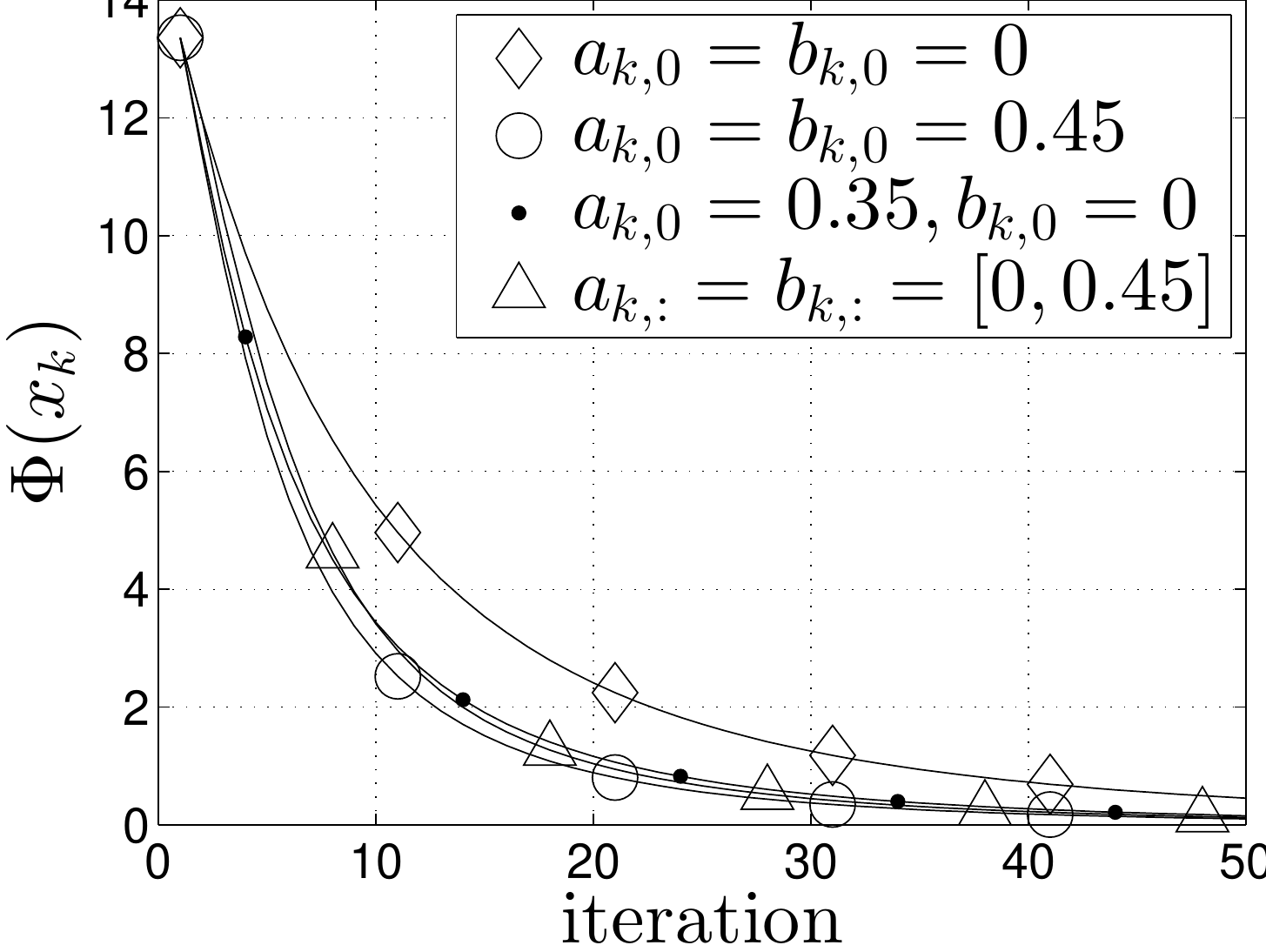}
\includegraphics[width=1.6in]{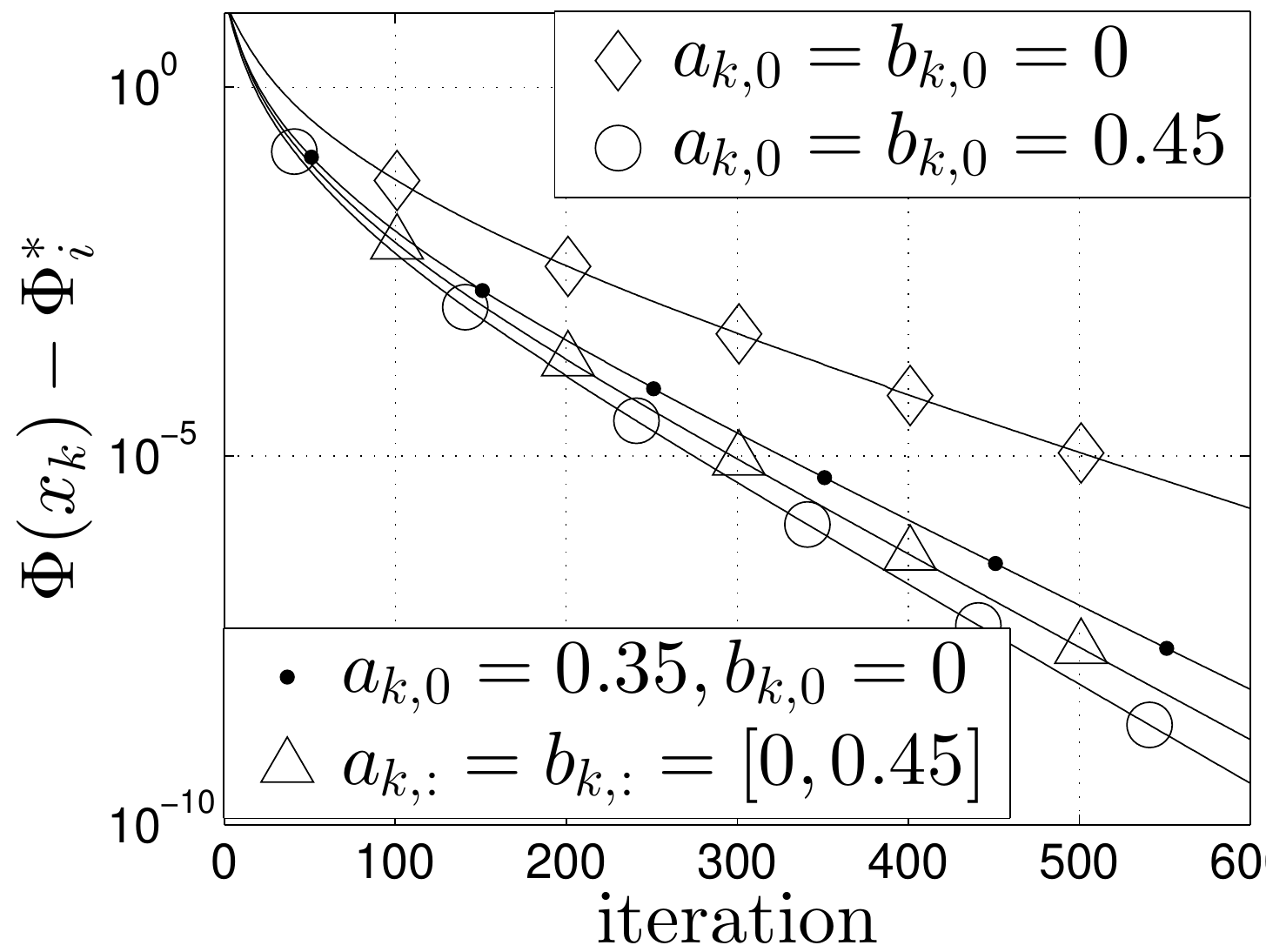}
\includegraphics[width=1.6in]{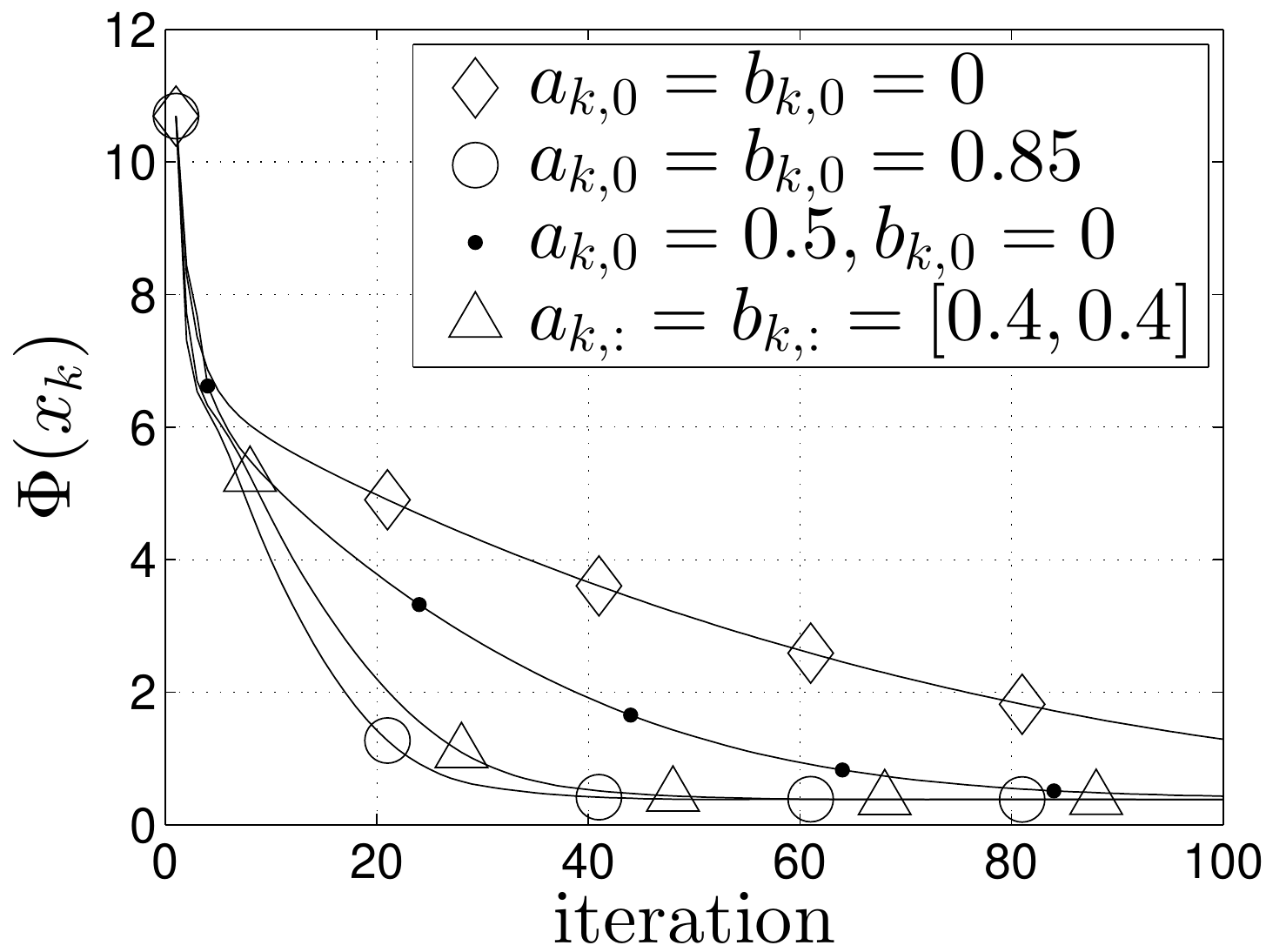}
\includegraphics[width=1.6in]{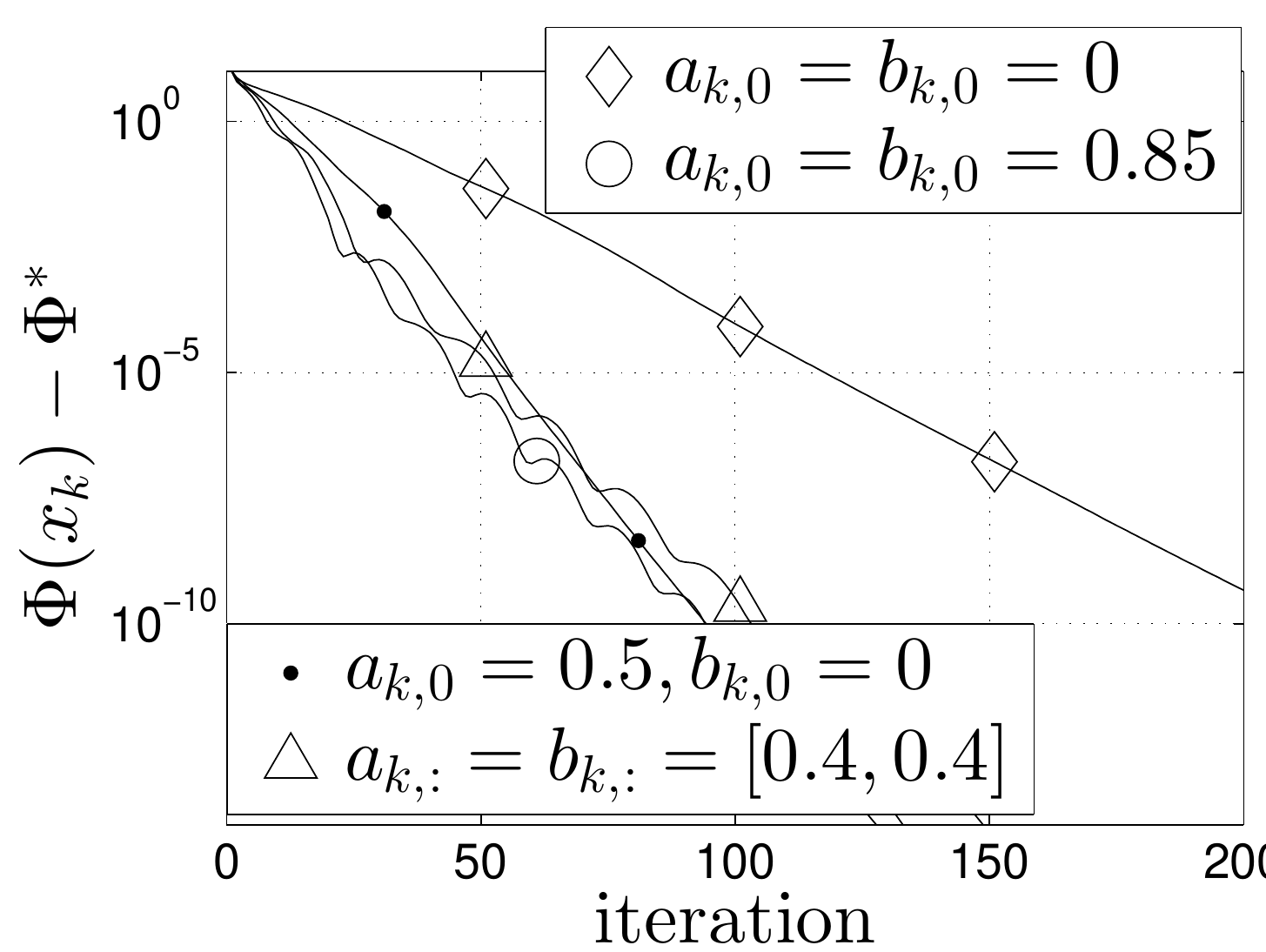}
% where an .eps filename suffix will be assumed under latex, 
% and a .pdf suffix will be assumed for pdflatex; or what has been declared
% via \DeclareGraphicsExtensions.
\caption{(Top Left) Plot of $\Phi(x_k)$ for SCAD least-squares. (Top Right) Plot of $\Phi(x_k)-\Phi_i^*$ with a logarithmic $y$-axis for SCAD least-squares. As SCAD least-squares is a nonconvex problem, each of the four considered parameter choices may converge to a different objective function value $\Phi_i^*$ for $i=1,2,3,4$. (Bottom Left) Plot of $\Phi(x_k)$ for $\ell_1$ least-squares. (Bottom Right) Plot of $\Phi(x_k)-\Phi^*$ with a logarithmic $y$-axis for $\ell_1$ least-squares.} 
\label{fig}
\end{figure}

We consider four valid parameter choices. To isolate the effect of inertia, all choices used the same randomly chosen starting point and fixed stepsize, $\gamma_k=0.1/L$ for SCAD and $\gamma_k=1/L$ for $\ell_1$. The inertial parameters were chosen so that $\delta>0$ (defined in (\ref{deltaDef})) for SCAD and to satisfy \cite[Thm. 1]{Johnstone2016Local} for the $\ell_1$ problem. The two figures on the right corroborate Theorem 2 in that all considered parameter choices converge linearly to their limit, which was estimated by using the attained objective function value after $1000$ iterations. For the nonconvex SCAD this is a new result. For $\ell_1$-regularized least squares, inertial methods have been shown to achieve \emph{local} linear convergence in \cite{Johnstone2016Local,liang2015activity} under additional strict complementarity or restricted strong convexity assumptions. However, our analysis, which is based on the KL inequality, does not explicitly require these additional assumptions, as the objective function always has a KL exponent of $1/2$ \cite[Lemma 10]{bolte2015error}. Furthermore our result proves \emph{global} linear convergence, in that the KL inequality (\ref{definitiveKL}) holds for all $k$, implying $k_0=1$ in (\ref{KLeq}) and (\ref{eq_linear}) holds for all $k$. In addition the two left figures show that the inertial choices appear to provide acceleration relative to the standard non-inertial choice which for SCAD is a new observation.  This does not conflict with Theorem 2 which only shows that both non-inertial and inertial methods will converge \emph{linearly}, however the convergence factor may be different. Estimating the factor is beyond the scope of this paper and we leave it for future work. Finally we mention that FISTA \cite{beck2009fast} and other Nesterov-accelerated methods \cite{nesterov2004introductory} are not applicable to SCAD as it is nonconvex.

\bibliographystyle{C:/Users/prjohns2.UOFI/Dropbox/Work_by_semester/ResearchF14on/ICASSP14/IEEEbib}
\bibliography{C:/Users/prjohns2.UOFI/Dropbox/Work_by_semester/ResearchF14on/ICASSP14/refs}

\end{document}